\documentclass[reqno,12pt]{amsart} 
\usepackage{amsmath,amssymb}

\newtheorem{theorem}{Theorem}
\newtheorem{proposition}[theorem]{Proposition}

\newtheorem{lemma}[theorem]{Lemma}

\theoremstyle{remark}

\newtheorem{definition}[theorem]{Definition}
\newtheorem{example}[theorem]{Example}

\numberwithin{equation}{section}

\newcommand{\qbin}{\genfrac{[}{]}{0pt}{}}
\newcommand{\C}{\mathbb C}
\newcommand{\R}{\mathbb R}
\newcommand{\N}{\mathbb N}

\allowdisplaybreaks

\author{Michael J.\ Schlosser}
\address{Fakult\"at f\"ur Mathematik, Universit\"at Wien,
Oskar-Morgenstern-Platz~1, A-1090 Vienna, Austria}
\email{michael.schlosser@univie.ac.at}

\author{Koushik Senapati}
\address{Fakult\"at f\"ur Mathematik, Universit\"at Wien,
Oskar-Morgenstern-Platz~1, A-1090 Vienna, Austria}
\email{koushik.bapan.19@gmail.com}

\author{Ali K.\ Uncu}
\address{Johann Radon Institute for Computational and Applied Mathematics,
Austrian Academy of Sciences,
Alten\-berger\-stra{\ss}e~69, A-4040 Linz, Austria}
\email{akuncu@risc.jku.at}

\title[Log-concavity of elliptic binomial coefficients]{Log-concavity results
for a biparametric and an elliptic extension of the $q$-binomial coefficients}

\dedicatory{Dedicated to Bruce Berndt, on the occasion of his 80th birthday.}



\begin{document}

\begin{abstract}
We establish discrete and continuous log-concavity results for a
biparametric extension of the $q$-numbers and of the $q$-binomial coefficients.
By using classical results for the Jacobi theta function we are able to lift
some of our log-concavity results to the elliptic setting.
One of our main ingredients is a putatively new lemma involving a
multiplicative analogue of Tur\'an's inequality.
\end{abstract}

\maketitle

\section{Introduction}\label{secintro}

In this paper we extend some well-known properties of the $q$-numbers and
$q$-binomial coefficients (a.k.a.\ \textit{Gau{\ss}ian binomial coefficients})
to settings involving extra parameters.
In particular, we establish discrete and continuous log-concavity results
for certain uniparametric, biparametric, and even elliptic extensions
of the $q$-numbers and $q$-binomial coefficients.

Given any complex $q\neq1$, the $q$-analogue of a complex number $x$
is defined by
\begin{equation*}
[x]_q:=\frac{1-q^x}{1-q}.
\end{equation*}
We refer to $[x]_q$ as a \textit{$q$-number} (or \textit{basic number}). 
The $q$-numbers play an important role in the theory of integer partitions
(see Andrews' book~\cite{Andrews} and historical references cited in there).
One can recover $x$ from $[x]_q$ by letting $q$ tend to $1$. 
Log-concavity results for $q$-binomial coefficients were given by
Butler~\cite{Butler}, Krattenthaler~\cite{Krattenthaler} and
Sagan~\cite{Sagan2}. (In this connection it should be mentioned
that the $q$-log-concavity or weighted log-concavity considered in the
literature implies log-concavity if the weights are non-negative.)
Recently Kalmykov and Karp~\cite{KalmykovKarp} established log-concavity
results (or equivalently, Tur\'an type inequalities) for specific basic
hypergeometric series.

For $a,b,q\in\C$ we define the $a,b;q$-extension
of a complex number $x$ as follows:
\begin{equation}\label{abqNum}
[x]_{a,b;q} :=
\frac{(1-q^x)(1-aq^x)(1-bq)(1-aq/b)}{(1-q)(1-aq)(1-bq^x)(1-aq^x/b)},
\end{equation}
where the variables are chosen such that none of the denominator factors vanish.
(This definition corresponds to that of the $a,b;q$-numbers considered
by the first author and Yoo in
\cite{SchlosserYoo1,SchlosserYoo2,SchlosserYoo3,SchlosserYoo5}.
In the latter $b$ has to be replaced by $bq^{-1}$ to match the definition
used in \eqref{abqNum}.)
Letting $a\to 0$ followed by $b\to 0$ (or $b\to 0$ followed by $a\to\infty$),
the $a,b;q$-numbers reduce to the basic numbers.

In Section~\ref{SectionDef} we list some elementary properties of the
$a,b;q$-numbers and explain the various notions of log-concavity we are
concerned about in this paper. Section~\ref{SectionNum} deals with results
about the log-concavity of the $a,b;q$-numbers. The lemma proved in that
section involves a multiplicative analogue of Tur\'an's inequality and plays
a key role in proving results involving $a,b;q$-numbers and $a,b;q$-binomial
coefficients, and in proving results in the (more general) elliptic setting.
Section~\ref{SectionBin} is devoted to log-concavity results for uni- and
biparametric extensions of the $q$-binomial coefficient. The
$a,b;q$-numbers can be further extended to the elliptic numbers that
appeared in \cite{SchlosserYoo1,SchlosserYoo2,SchlosserYoo3,SchlosserYoo5},
and are the contents of study in Sections~\ref{sec:ell-prim} and
\ref{sec:ell-results}.
Our elliptic numbers are indeed elliptic functions (i.e., they are meromorphic
and doubly periodic); they are expressed as certain ratios of theta functions.
Accordingly, the analysis in Sections~\ref{sec:ell-prim} and
\ref{sec:ell-results} involves some machinery from the theory of Jacobi
theta functions (or, equivalently, of the Weierstra{\ss} sigma function)
which is classical but not so well-known in the community of $q$-series,
which is the reason why we cover this material in separate sections of our
paper. Finally, in Section~\ref{sec:con} we conclude with an outlook of
further open problems.

\section{Preliminaries}\label{SectionDef}

It is a matter of simple algebra to verify for arbitrary
$x$ and $y$ the following addition formula
for the $a,b;q$-numbers
defined in \eqref{abqNum}:
\begin{subequations}\label{abq_addition}
\begin{equation}
[x]_{a,b;q} + W_{a,b;q}(x) [y-x]_{aq^{2x},bq^x;q} = [y]_{a,b;q},\end{equation}
where $W_{a,b;q}(x)$ is the \textit{$a,b;q$-weight}, defined by
\begin{equation}\label{single_weight}
W_{a,b;q}(x) = \frac{(1-aq^{1+2x})(1-b)(1-bq)(1-a/b)(1-aq/b)}
{(1-aq)(1-bq^x)(1-bq^{1+x})(1-aq^x/b)(1-aq^{1+x}/b)}q^x.
\end{equation}
\end{subequations}

Now if we impose $0<q<1$ and $0<a<b<1$
(in particular, all variables are presumed to be real),
it is easy to see that for any real $x>0$ we have
\begin{equation}\label{positivity_dom}
[x]_{a,b;q}>0\qquad and\qquad W_{a,b;q}(x)>0,
\end{equation}
as all the factors appearing in the respective quotients
are manifestly positive.

It is also easy to observe the following three properties of the
$a,b;q$-numbers and the associated $a,b;q$-weights:
\begin{subequations}\label{properties}
\begin{align}
\label{zero_property}[0]_{a,b;q} &= 0\qquad and\qquad W_{a,b;q}(0)=1,\\
\label{order_rel}[x]_{a,b;q} &\geq [y]_{a,b;q}\qquad
\text{for}\quad 0<q<1,\quad x\geq y\ge 0, \quad\text{and}\quad 0<a<b<1,\\
\label{negative_numbers_def}[x]_{a,b;q}  &= - W_{a,b;q}(x)\,[-x]_{aq^{2x},bq^x;q}.
\end{align}
\end{subequations}

The relation \eqref{order_rel} follows from the addition formula
in \eqref{abq_addition}, with $x$ and $y$ having been interchanged,
as for $x>y>0$ the difference $[x]_{a,b;q}-[y]_{a,b;q}$ is
$W_{a,b;q}(y) [x-y]_{aq^{2y},bq^y;q}$ which is positive by 
\eqref{positivity_dom}.
While we could use the above relation \eqref{negative_numbers_def}
to deal with the $a,b;q$-numbers of negative argument,
in this paper we shall restrict our attention to the case that the
arguments are non-negative real numbers.

There are two intermediate extensions from the basic-numbers
to the $a,b;q$-numbers. These two intermediate extensions correspond to the
limits $b\to 0$, and to $a\to 0$, in the $a,b;q$-numbers, respectively.
Specifically, one can let $b\to 0$ (or $b\to\infty$) in \eqref{abqNum},
by which one obtains the \textit{$a;q$-numbers}
(studied in \cite{SchlosserYoo2,SchlosserYoo4,SchlosserYoo5}):
\begin{equation}\label{a_q_numbers}
[x]_{a;q}:=[x]_{a,0;q} = \frac{(1-q^x)(1-aq^x)}{(1-q)(1-aq)}q^{1-x}.
\end{equation}
These do not only generalize the standard $q$-numbers $[x]_q$
obtained by letting $a\to 0$ in \eqref{a_q_numbers},
but also the \textit{quantum} numbers
$\langle x\rangle_q:=(q^x-q^{-x})/(q-q^{-1})$
(which frequently appear in physical models),
obtained by letting $a\to -1$ in \eqref{a_q_numbers}.

One can also let $a\rightarrow 0$ (or $a\to\infty$) in \eqref{abqNum}
and arrive at the \textit{$(b;q)$-numbers}
\begin{equation}\label{b_q_numbers}
[x]_{(b;q)}:=[x]_{0,b;q} = \frac{(1-q^x)(1-bq)}{(1-q)(1-bq^x)}.
\end{equation}
We decided to put parantheses in ``$(b;q)$-numbers'' but none in
``$a;q$-numbers'' to distinguish them in notation, thus to avoid confusion.
(For instance, we have $[x]_{(0;q)}=[x]_q$ but $[x]_{0;q}=[x]_{q^{-1}}$.)

In terms of standard terminology for basic hypergeometric series
(cf.\ \cite{GasperRahman}), the basic hypergeo\-met\-ric expression
on the right-hand side of \eqref{a_q_numbers} is \textit{well-poised}
and that on the right-hand side of \eqref{b_q_numbers} is \textit{balanced}.
It should come to no surprise that the right-hand side of
\eqref{abqNum} is well-poised and balanced (while the corresponding
expression for the weight in \eqref{single_weight} is even
\textit{very-well-poised} and balanced).

We now explain different notions of log-concavity.
\begin{definition}\label{log_concavity}
A sequence of real numbers
$(a_k)_{k=0}^{\infty}$ (indexed by non-negative integers)
is called \textit{log-concave} if
\begin{equation}\label{log_conc_def}
a_k^2 \geq a_{k+1}a_{k-1},
\end{equation} for all $k\geq 1$.
Similarly, one calls a sequence $(a_k)_{k=0}^{\infty}$
\textit{strongly log-concave} if
\begin{equation}\label{st_log_conc_def} a_k a_l \geq a_{k+1}a_{l-1}\end{equation}
for all positive integers $k$ and $l$ with $k\geq l$.
\end{definition}
It is clear that if the $a_k$ are all positive (or all negative),
log-concavity implies strong log-concavity since \eqref{log_conc_def} then
is equivalent to
\begin{equation*}
\frac{a_k}{a_{k+1}}\ge\frac{a_{k-1}}{a_k},
\end{equation*}
which can be iterated to establish \eqref{st_log_conc_def}.

We will also use the notions of log-concavity and strong log-concavity
in the continuous setting.
\begin{definition}\label{log_concavity-func}
A function $a(x)$ depending on a non-negative real variable $x$
is called continuously log-concave if
\begin{equation}\label{log_conc_def-func}
a(x)^2 \geq a(x+r)a(x-r),
\end{equation}
for all $x\geq r\ge 0$, and continuously
strongly log-concave if
\begin{equation}\label{st_log_conc_def-func}
a(x) a(y) \geq a(x+r)a(y-r)\end{equation}
for all real $x\geq y\geq r\ge 0$.
\end{definition}
Again it is easy to see that if $a(x)>0$ for all $x\ge 0$ (or
$a(x)<0$ for all $x\ge 0$) log-concavity implies strong log-concavity
since the down-shift by $r$ of the arguments,
$a(x)/a(x+r)\ge a(x-r)/a(x)$, can be iterated with an additional
down-shift by $x-y$ to establish \eqref{st_log_conc_def-func}.

For a non-vanishing positive (or non-vanishing negative)
function $a(x)$, $x\ge 0$, one can equivalently express the continuous
strong log-concavity as 
\begin{equation}\label{log_conc_equivalent}
\frac{a(x+r)a(y-r)}{a(x) a(y)} \leq 1,
\end{equation}
where $x\geq y\geq r\geq 0$.
Similar reformulations can be applied to the other notions of
log-concavity considered above and we will be using them when convenient.

\begin{example}
Maybe the most trivial example of a continuously strong log-concave function
is the identity on $[0,\infty)$. Indeed, assuming $x\geq y\geq r>0$
(the case $r=0$ of \eqref{st_log_conc_def-func} is trivial),
\begin{equation}\label{eq:triv}
\frac{(x+r)(y-r)}{xy}<1
\end{equation}
of course holds, since $xy-(x+r)(y-r)=r(r+x-y)>0$.
This simple fact is already responsible for the continuous strong log-concavity
of the \textit{continuous binomial coefficients} (which were recently studied
by Salwinski~\cite{Salwinski} who proved identities satisfied by them,
among them also a continuous binomial theorem),
defined by
\begin{equation}\label{cont-bin1}
\binom xk=\frac{\Gamma(1+x)}{\Gamma(1+k)\Gamma(1+x-k)}
\qquad\text{for $x,k\in\C$, $x\notin -1,-2,\ldots$.}
\end{equation}
By virtue of Euler's product formula for the gamma function,
\begin{equation}
\Gamma(1+x)=\prod_{j=1}^\infty\frac{j^{1-x}(1+j)^x}{x+j}
\qquad\text{for $x\in\C$, $x\notin -1,-2,\ldots$},
\end{equation}
we may rewrite Equation~\eqref{cont-bin1} in the following convenient product form:
\begin{equation}
\binom xk=\prod_{j=1}^\infty\frac{(k+j)(x-k+j)}{j(x+j)}
\qquad\text{for $x,k\in\C$, $x\notin -1,-2,\ldots$.}
\end{equation}
It is now easy to deduce the following result:
\textit{For any real $x,y,k,l,r$ satisfying $x\geq y$,
$k\geq l\ge r\ge 0$, and $y-l\geq x-k$, we have the continuous
strong log-concavity}
\begin{equation}\label{binom_inequality}
\binom{x}{k}\binom{y}{l} \geq \binom{x}{k+r}\binom{y}{l-r}.
\end{equation}
\begin{proof}
The $r=0$ case is trivial, so assume $r>0$. After canceling common factors
we see that
\begin{equation*}
\frac{\binom x{k+r}\binom y{l-r}}{\binom x{k}\binom y{l}}=
\prod_{j=1}^\infty\frac{(k+r+j)(x-k-r+j)(l-r+j)(y-l+r+j)}{(k+j)(x-k+j)(l+j)(y-l+j)}<1,
\end{equation*}
since by taking different instances of \eqref{eq:triv}, we have
\begin{equation*}
\frac{(k+r+j)(l-r+j)}{(k+j)(l+j)}<1,\quad\;\text{and}\quad\,
\frac{(y-l+r+j)(x-k-r+j)}{(y-l+j)(x-k+j)}<1,
\end{equation*}
for each $j\geq 1$, which establishes the claim.
\end{proof}
\end{example}

\section{Log-concavity of $a,b;q$-numbers}\label{SectionNum}

To deal with the continuous log-concavity
of the $a;q$-, the $(b;q)$- and the $a,b;q$-numbers we will make use
of the following elementary result.

\begin{proposition}\label{Proposition_1} For $0<q<1$ and $0\leq \nu <1$
the function $x\mapsto (1-\nu q^x)$, $x\ge 0$
is continuously strongly log-concave.
\end{proposition}

\begin{proof} Assume that $x \geq y \geq r\geq 0$.
Using the definition \eqref{st_log_conc_def-func}
directly and rearranging the terms yields
\begin{equation*}
(1-\nu q^x)(1-\nu q^y)
-(1-\nu q^{x+r})(1-\nu q^{y-r})= \nu q^{y-r}(1-q^r)(1-q^{x-y+r})\ge 0.
\qedhere
\end{equation*}
\end{proof}

In particular, we can utilize Proposition~\ref{Proposition_1}
to establish the following result involving the $a;q$-numbers.

\begin{theorem}\label{a_q_num_log_concave}
For $0<q<1$ and $0<a<1$
the $a;q$-numbers $[x]_{a;q}$ are continuously strongly log-concave.
\end{theorem}

\begin{proof}
Let $x\geq y\geq r\geq 0$. Using \eqref{log_conc_equivalent}
we have to show that the fraction
\begin{equation*}
\frac{[x+r]_{a;q}[y-r]_{a;q}}{[x]_{a;q}[y]_{a;q}} =
\frac{(1-q^{x+r})(1-q^{y-r})}{(1-q^x)(1-q^y)}
\frac{(1-aq^{x+r})(1-aq^{y-r})}{(1-aq^x)(1-aq^y)}
\end{equation*}
is less or equal to $1$.
Now the two fractions on the right-hand side are both non-negative
and less than or equal to $1$ by virtue of Proposition~\ref{Proposition_1}
(with $\nu=1$ and $\nu=a$), and \eqref{log_conc_equivalent}.
Thus their product is less than or equal to $1$.
\end{proof}

The same argument is not applicable to the $(b;q)$-numbers
$[x]_{(b;q)}$, nevertheless we can show that the $(b;q)$-numbers
possess the continuously strong log-concavity property by elementary algebra. 

\begin{theorem}\label{b_q_num_log_concave}
For $0<q<1$ and $0<b<1$ the $(b;q)$-numbers
$[x]_{(b;q)}$ are continuously strongly log-concave.
\end{theorem}

\begin{proof} Let $x\geq y\geq r\geq 0$. By direct computation,
we see that 
\begin{align*}
&[x]_{(b;q)}[y]_{(b;q)}-[x+r]_{(b;q)}[y-r]_{(b;q)}\\
&=\frac{(1-q^r)(1-q^{x-y+r})(1-b)(1-bq)^2(1-bq^{x+y})}
{(1-q)^2(1-bq^x)(1-bq^{x+r})(1-bq^{y})(1-bq^{y-r})}q^{y-r}.
\end{align*} 
All the factors appearing in the fraction on the right-hand side
are non-negative. Thus the whole product is non-negative, so
\eqref{st_log_conc_def-func} holds.
\end{proof}

We now move our attention to the $a,b;q$-numbers.
Before we show the continuous strong log-concavity of these numbers according
to  Definition~\ref{log_concavity-func}, we show that these numbers
satisfy an $a,b;q$-version of continuous strong log-concavity
where certain shifts of the variables $a$ and $b$ occur.

\begin{theorem}\label{log_concavity_abq_addition}
For all real $x,y,a,b,q$ satisfying $0<q<1$, $0<a<b<1$, and
$x\geq y\geq r\geq 0$, the $a,b;q$-numbers $[x]_{a,b;q}$
satisfy the inequality 
\begin{equation*}
[x]_{aq^{2r},bq^r;q}[y]_{a,b;q} \geq [x+r]_{a,b;q}[y-r]_{aq^{2r},bq^r;q}.
\end{equation*}
\end{theorem}

\begin{proof}
Using \eqref{abq_addition}, we have
\begin{align*}
&[x]_{aq^{2r},bq^r;q}[y]_{a,b;q} - [x+r]_{a,b;q}[y-r]_{aq^{2r},bq^r;q}\\
&=[x]_{aq^{2r},bq^r;q}
\big([r]_{a,b;q}+W_{a,b;q}(r)[y-r]_{aq^{2r},bq^r;q}\big)\\
  &\qquad- \big([r]_{a,b;q}+
    W_{a,b;q}(r)[x]_{aq^{2r},bq^r;q}\big)[y-r]_{aq^{2r},bq^r;q}\\
&=[r]_{a,b;q}\big([x]_{aq^{2r},bq^r;q} - [y-r]_{aq^{2r},bq^r;q}\big)\\
&=[r]_{a,b;q}W_{aq^{2r},bq^r;q}(y-r)[x-y+r]_{aq^{2y},bq^y;q}.
\end{align*}
The non-negativity of the last expression is clear from \eqref{positivity_dom}. 
\end{proof}

We can actually avoid using shifts of the variables $a$ and $b$
and show (pure) continuous strong log-concavity of the $a,b;q$-numbers.

\begin{theorem}\label{direct_log_concavity}
The $a,b;q$-numbers $[x]_{a,b;q}$ are continuously
strongly log-concave. In particular, for all real $x,y,a,b,q$ satisfying $0<q<1$, $0<a<b<1$, and $x\geq y\geq r\geq 0$, we have
\begin{equation*}
[x]_{a,b;q}[y]_{a,b;q}\geq [x+r]_{a,b;q}[y-r]_{a,b;q}.
\end{equation*}
\end{theorem}

Theorem~\ref{direct_log_concavity} is a direct consequence
of a general lemma, see Lemma~\ref{lem:mult_Turan} below,
which involves a multiplicative analogue of Tur\'an's inequality.
In its formulation, the role of the variable $\delta$,
which is assumed to be a fixed non-negative real number, at first
may seem to be enigmatic. Indeed, for the sole purpose of proving
Theorem~\ref{direct_log_concavity}, the case $\delta=0$ would
completely suffice. However, we shall also employ the lemma
when dealing with theta functions in the proofs of
Theorems~\ref{direct_log_concavity_ell} and \ref{binom_top_log_concavity_ell}
and there we require $\delta$ to be a specific positive real number.

\begin{lemma}\label{lem:mult_Turan}
Let $0\leq\delta<\lambda$.
Let $f$ be a continuous positive real function on the
interval $[\delta,\lambda]$, twice differentiable
with values of first and second derivatives being
negative on the interval $(\delta,\lambda)$. Let $\delta< a\leq b<\lambda$.
Then we have the inequality
\begin{equation}\label{mult_Turan}
\frac{f(\lambda)f(a)}{f(b)f(\lambda a/b)} \leq 1.
\end{equation}
\end{lemma}

\begin{proof} Without loss of generality, assume that $b \geq \lambda a/b$;
if not, one can switch $b$ and $\lambda a/b$.
Let $x= f(\lambda a/b)-f(\lambda)$, $y = f(b) -f(\lambda a/b)$, and
$z = f(a) - f(b)$.
Using these new variables the claimed inequality 
\eqref{mult_Turan} is equivalent to
\begin{equation*}
f(\lambda) (z-x) - x(x+y)\leq 0.
\end{equation*}
The $z\leq x$ case is clear (since $x\geq 0$ and $x+y>0$)
while the case $z>x$ is actually vacuous.
Indeed, $z>x$, i.e. $f(a)-f(b)>f(\lambda a/b)-f(\lambda )$,
would be equivalent to
\begin{equation*}
\frac{f(a)-f(\lambda a/b)}{\lambda-b}>\frac{f(b)-f(\lambda)}{\lambda-b},
\end{equation*}
which again is equivalent to
\begin{equation}\label{feq}
  \frac ab\cdot\frac{f(\frac{\lambda a}b)-f(a)}
  {\frac{\lambda a}b-a}<\frac{f(\lambda)-f(b)}{\lambda-b}.
\end{equation}
Recall that $a<\lambda a/b\le b<\lambda$.
By two applications of the mean value theorem there exist
$c\in(a,\lambda a/b)$ and $d\in(b,\lambda)$ (in particular,
$c<d$ must hold), such that
\begin{equation*}
  f'(c)=\frac{f(\frac{\lambda a}b)-f(a)}{\frac{\lambda a}b-a}
  \qquad\text{and}\qquad
f'(d)=\frac{f(\lambda)-f(b)}{\lambda-b}.
\end{equation*}
Since, by assumption, the first derivatives of $f$ are negative in
$(\delta,\lambda)$, we obtain from \eqref{feq} the following string
of inequalities:
\begin{equation*}
\frac ab f'(c)<f'(d)<0.
\end{equation*}
If we substract $f'(c)$ from both sides of the first inequality, and
subsequently divide by $d-c$, we obtain
\begin{equation*}
\left(\frac ab-1\right) \frac{f'(c)}{d-c}<\frac{f'(d)-f'(c)}{d-c}.
\end{equation*}
Now since $\frac ab-1<0$ and $f'(c)<0$ (by assumption), the product 
on the left-hand side of the inequality is positive.
But the expression on the right-hand side is negative,
as by the mean value theorem it is equal to the second derivative
of $f$ at some point in $(c,d)$,
which by assumption must be negative. This is a contradiction,
thus the lemma is proved.
\end{proof}

Now we are ready to prove Theorem~\ref{direct_log_concavity}.
\begin{proof}[Proof of Theorem~\ref{direct_log_concavity}]
By the remarks following Definition~\ref{log_concavity-func} it suffices
to show the continuous log-concavity of the $a,b;q$-numbers
\eqref{abqNum}, i.e.,
\begin{equation}\label{abq_log_conc}
\frac{[x+r]_{a,b;q}[x-r]_{a,b;q}}{[x]_{a,b;q}^2}\leq 1.
\end{equation}
Using the definition of the $a,b;q$-numbers \eqref{abqNum},
after some cancellations, we see that the left-hand side of
\eqref{abq_log_conc} is
\begin{align}\label{long_fractions}
&\frac{(1-q^{x+r})(1-q^{x-r})}{(1-q^x)^2}
\frac{(1-aq^{x+r})(1-aq^{x-r})}{(1-aq^x)^2}\notag\\
&\times\frac{(1-bq^x)^2}{(1-bq^{x+r})(1- bq^{x-r})}
\frac{(1-aq^x/b)^2}{(1-aq^{x+r}/b)(1-aq^{x-r}/b)}.
\end{align}
For a fixed pair of real numbers $x\geq r> 0$, let
\begin{equation}\label{f_func} f(u) := f_{q,x,r}(u) =
\frac{(1-uq^{x+r})(1-uq^{x-r})}{(1-uq^x)^2}\qquad\text{for $u\in[0,1]$}.
\end{equation}
We have
\begin{equation*}
f'(u)=-\frac{(1-q^r)^2(1+uq^x)}{(1-uq^x)^3}q^{x-r}<0,
\end{equation*}
and
\begin{equation*}
f''(u)=-\frac{2(1-q^r)^2(2+uq^x)}{(1-uq^x)^4}q^{2x-r}<0,
\end{equation*}
for any $u\in (0,1)$, 
which allows us to apply Lemma~\ref{lem:mult_Turan} with $f$ as defined
in \eqref{f_func}, with $\delta=0$ and $\lambda=1$.
This immediately establishes that the expression in
\eqref{long_fractions} is $\leq 1$.
\end{proof}

\section{$a,b;q$-Binomial coefficients and
log-concavity}\label{SectionBin}

In this section, we recall the definition of the $a,b;q$-analogue of binomial
coefficients (that first appeared in work of the first author~\cite{Schlosser1}
in the context of enumeration of weighted lattice paths and subsequently were
studied in further work of the first author~\cite{Schlosser2} in connection
with a non-commutative binomial theorem, and in joint work of the first author
with Yoo~\cite{SchlosserYoo1, SchlosserYoo2}; in particular, these papers
contain their recurrence relations and combinatorial interpretations,
also for the more general $a,b;q,p$- or elliptic binomial coefficients,
which include the $a,b;q$-binomial coefficients obtained by letting $p\to 0$).

For the following definition we restrict the base $q$ to satisfy $0<|q|<1$
(while for studying inequalities we further assume $q$ to be real and
satisfy $0<q<1$). For parameter $a\in\C$ and lower index $k\in\C$,
the \textit{$q$-shifted factorial} is defined as
\begin{equation*}
(a;q)_k:=\frac{(a;q)_\infty}{(aq^k;q)_\infty}, \qquad\text{where}\quad
(a;q)_\infty=\prod_{j\geq 0}(1-aq^j)
\end{equation*}
cf.\ \cite{GasperRahman}. (For $a=q^{-n}$ with integers $n\geq k\geq 0$
the above definition for $(a;q)_k$ involves a pole which however can be
removed.)

For brevity, we use the following compact notation for products of
$q$-shifted factorials:
\begin{equation*}
(a_1,\dots,a_m;q)_k:=(a_1;q)_k\cdots(a_m;q)_k,\qquad\text{where}\quad
k\in\C\cup\{\infty\}.
\end{equation*}

For $x,k,a,b,q\in\C$, $0<|q|<1$, we define the $a,b;q$-binomial coefficient as
\begin{align}\label{a_b_q_binom}
\qbin{x}{k}_{a,b;q} :={} &\frac{(q^{1+k},aq^{1+k},bq^{1+k},aq^{1-k}/b;q)_{x-k}}
{(q,aq,bq^{1+2k},aq/b;q)_{x-k}}\\\notag
={}&\frac{(q^{1+x-k},aq^{1+x-k},bq^{1+k},aq^{1-k}/b;q)_k}
{(q,aq,bq^{1+x},aq^{1+x-2k}/b;q)_k}.
\end{align}
(Note that the $a,b;q$-binomial coefficients are \textit{not}
symmetric in $k$ and $x-k$, contrary to the symmetry for the
ordinary and basic binomial coefficients.)
If we assume $x$ and $k$ to be non-negative integers, and
replace $b$ by $bq$ on the right-hand side, then
definition \eqref{a_b_q_binom}
matches that of the $a,b;q$-binomial coefficient as used
in \cite{Schlosser1,Schlosser2,SchlosserYoo1,
  SchlosserYoo2,SchlosserYoo3,SchlosserYoo5}.

If we let $k=1$ and divide $b$ by $q$, the  $a,b;q$-binomial coefficient
reduces to the $a,b;q$-number defined in \eqref{abqNum}:
\begin{equation*}
\qbin{x}{1}_{a,b/q;q}=[x]_{a,b;q}.
\end{equation*}

We define the $a;q$ and $(b;q)$ analogues of binomials by taking limits
$b\rightarrow\infty$ and $a\rightarrow 0$
of \eqref{a_b_q_binom}, respectively. Then
\begin{equation}\label{aq_bq_binom}
\qbin{x}{k}_{a;q} := \frac{(q^{1+k},aq^{1+k};q)_{x-k}}
{(q,aq;q)_{x-k}} q^{k(k-x)} ,\quad\;\text{and}\quad\;
\qbin{x}{k}_{(b;q)} := \frac{(q^{1+k},bq^{1+k};q)_{x-k}}
{(q,bq^{1+2k};q)_{x-k}}.
\end{equation}
(While the $a;q$-binomial coefficients are symmetric in $k$ and $x-k$,
the $(b;q)$-binomial coefficients are not.) It is clear that
\begin{equation*}
\qbin{x}{1}_{a;q}=[x]_{a;q}\qquad\text{and}
\qquad\qbin{x}{1}_{(b/q;q)}=[x]_{(b;q)},
\end{equation*}
with the $a;q$- and $(b;q)$-numbers defined in \eqref{a_q_numbers} and
\eqref{b_q_numbers}, respectively.

Parallel to Section~\ref{SectionNum}, we start our
log-concavity discussion with the $a;q$-binomial coefficients.
For these we have continuous strong log-concavity with respect to
the \textit{lower} parameter of the binomial coefficient.

\begin{theorem}\label{aq_binomial_log_concavity_1}
For any real $a,x,y,k,l,r,q$ satisfying $0<q<1$, $0\leq a<1$, $x\geq y$,
$k\geq l\ge r\ge 0$, and $y-l\geq x-k$, we have the continuous
strong log-concavity
\begin{equation}\label{aq_binom_inequality}
\qbin{x}{k}_{a;q}\qbin{y}{l}_{a;q} \geq \qbin{x}{k+r}_{a;q}\qbin{y}{l-r}_{a;q}.
\end{equation}
\end{theorem}

\begin{proof}
After cancellations, we see that
\begin{align*}
\frac{\qbin{x}{k+r}_{a;q}\qbin{y}{l-r}_{a;q}}
{\qbin{x}{k}_{a;q}\qbin{y}{l}_{a;q}}&=q^{2r(r+k-l)}
\frac{(q^{1+x-k-r},aq^{1+x-k-r},q^{1+k+r},aq^{1+k+r};q)_\infty}
{(q^{1+x-k},aq^{1+x-k},q^{1+k},aq^{1+k};q)_\infty}\\*
&\quad\;\times
\frac{(q^{1+y-l+r},aq^{1+y-l+r},q^{1+l-r},aq^{1+l-r};q)_\infty}
{(q^{1+y-l},aq^{1+y-l},q^{1+l},aq^{1+l};q)_\infty}.
\end{align*}
Now, $q^{2r(r+l-l)}\leq 1$ and by taking different instances of
Proposition~\ref{Proposition_1}, we have
\begin{equation}\label{eq:nu}
\frac{(1-\nu q^{k+r})(1-\nu q^{l-r})}{(1-\nu q^{k})(1-\nu q^{l})}\leq 1,
\quad\;\text{and}\quad\;
\frac{(1-\nu q^{y-l+r})(1-\nu q^{x-k-r})}{(1-\nu q^{y-l})(1-\nu q^{x-k})}\leq 1,
\end{equation}
for $\nu=q^j$ and $\nu=aq^j$, where $j=1,2,\dots$.
Thus the infinite product is also $\le 1$ and the theorem follows.
\end{proof}

We note that rhe $(b;q)$-binomial coefficients and $a,b;q$-binomial
coefficients do not appear to be log-concave with respect to their
lower parameter.

We now turn back to the $a,b;q$-binomial coefficients.
We are able to prove the discrete strong log-concavity of
$\qbin{x}{k}_{a,b;q}$
(here $x$ may be real but $k$ should be a non-negative integer) 
with respect to to the upper parameter $x$
but require that the range of the variables $a,b$ depends
on the lower parameter $k$.

\begin{theorem}\label{binom_top_log_concavity}
For any real $a,b,x,y,q$ and non-negative integer $k$ satisfying
$0<q<1$, $0<a\leq bq^k<1$, and $x\geq y\ge 1$, with the difference $x-y$ being
a non-negative integer, we have the discrete strong log-concavity
\begin{equation}\label{abq_binom_inequality}
\qbin{x}{k}_{a,b;q}\qbin{y}{k}_{a,b;q} \geq
\qbin{x+1}{k}_{a,b;q}\qbin{y-1}{k}_{a,b;q}.
\end{equation}
\end{theorem}

\begin{proof}
All we need to show is that
\begin{equation*}
\frac{\qbin{x+1}{k}_{a,b;q}\qbin{x-1}{k}_{a,b;q}}{\qbin{x}{k}^2_{a,b;q} }\leq 1,
\end{equation*}
as then the iteration of the inequality
\begin{equation*}
\frac{\qbin{x}{k}_{a,b;q}}{\qbin{x+1}{k}_{a,b;q} }\geq
\frac{\qbin{x-1}{k}_{a,b;q}}{\qbin{x}{k}_{a,b;q} }
\end{equation*}
leads to \eqref{abq_binom_inequality}.
After canceling common factors, we obtain
\begin{equation}\label{moving_f_to_g}
\frac{\qbin{x+1}{k}_{a,b;q}\qbin{x-1}{k}_{a,b;q}}{\qbin{x}{k}^2_{a,b;q} }=
  \frac{g(1)g(a)}{g(bq^k)g(a/bq^k)},
\end{equation} 
where
\begin{equation*}
g(u):= g_{q,x,k}(x) = \frac{(1-uq^{x-k})(1-uq^{x+1})}{(1-uq^{x-k+1})(1-uq^x)}.
\end{equation*}
Observe that 
\begin{equation*}
g_{q,x,k}(u) = \prod_{i=1}^{k} f_{q,x-k+i,1}(u),
\end{equation*}
where $f_{q,x,1}(u)$ is as defined in \eqref{f_func}. 
Therefore, the right-side of \eqref{moving_f_to_g} can be rewritten as
\begin{equation*}
  \frac{\qbin{x+1}{k}_{a,b;q}\qbin{x-1}{k-1}_{a,b;q}}{\qbin{x}{k}^2_{a,b;q} }
  =\prod_{j=1}^{k}\frac{f_{q,x-k+j,1}(1)f_{q,x-k+j,1}(a)}
  {f_{q,x-k+j,1}(bq^k)f_{q,x-k+j,1}(a/bq^k)}.
\end{equation*}
The $k$ terms of the product are each less than or equal to $1$
by the $\delta=0$, $\lambda=1$ and
$b\mapsto bq^k$ case of Lemma~\ref{mult_Turan} with $f(u)$
as defined in \eqref{f_func} but with $b$ replaced by $bq^k$.
\end{proof}

This is consistent with the log-concavity of the numbers $[n]_{a,b;q}$
that we proved in Theorem~\ref{direct_log_concavity}; that is namely just
the $k=1$ and $b\mapsto b/q$ case of Theorem~\ref{binom_top_log_concavity}.

\section{Elliptic numbers and elliptic binomial coefficients}
\label{sec:ell-prim}

We start with explaining some standard notions about theta functions and
elliptic functions. A function $g$ of a complex variable $z$
is \textit{elliptic} if it is meromorphic and doubly periodic.
It follows from the theory of abelian functions (cf.\ \cite{Weber})
that there exists a non-negative integer $s$ and complex numbers
$a_1,\dots a_s$, $b_1,\dots,b_s$,  $c$ and $p$ with $0<|p|<1$ such that
\begin{equation*}
g(z)=c\,\frac{\theta(a_1z,\dots,a_sz;p)}{\theta(b_1z,\dots,b_sz;p)},
\qquad\text{where}\quad a_1\cdots a_s= b_1\cdots b_s\neq0.
\end{equation*}
Here, $\theta(x_1,\dots,x_s;p):=\theta(x_1;p)\cdots\theta(x_s;p)$,
where, for $x\neq 0$ and $|p|<1$,
\begin{equation*}
\theta(x;p)=(x,p/x;p)_\infty
\end{equation*}
is the \textit{modified Jacobi theta function} (which in short we refer
to as \textit{theta function}) of argument $x$ and \textit{nome} $p$.
It is clear that for $p=0$ we have $\theta(x;0)=(1-x)$.
So the theta function can be considered to be a $p$-extension of a
linear factor. The following properties hold for the theta function:
\begin{subequations}\label{eq:theta-easy}
\begin{align}
\theta(x;p)&=-x\theta(1/x;p),\\
\theta(px;p)&=-\frac 1x\theta(x;p),
\end{align}
\end{subequations}
and the \textit{addition formula}
\begin{equation}\label{eq:theta-addition}
\theta(xy,x/y,ut,u/t;p)-\theta(xt,x/t,uy,u/y;p)=
\frac uy\theta(yt,y/t,xu,x/u;p).
\end{equation}
The two identities in \eqref{eq:theta-easy}, which can be referred to as
\textit{inversion formula} and \textit{quasi-periodicity},
are readily shown from the definition of the theta function,
while the theta function addition formula
(originally due to Weierstra{\ss}, see \cite[p.~451, Example 5]{WW})
in \eqref{eq:theta-addition} is not obvious but crucial in the theory of
elliptic hyper\-geometric functions.
(See e.g.\ Rosengren's lecture notes \cite[Sec.~1.4]{Rosengren}
for the standard derivation of \eqref{eq:theta-addition}
involving complex analysis.)

We recall further classical facts about theta functions which
we need. In particular, we would like to have formulas involving their
first and second derivatives, so we can apply Lemma~\ref{lem:mult_Turan}
with $f$ involving a quotient of theta functions, rather than a quotient
of linear factors as in \eqref{f_func}. This will enable us to extend
Theorems~\ref{direct_log_concavity} and \ref{binom_top_log_concavity}
to the elliptic setting.

The modified Jacobi theta function is a special case of a
\textit{sigma function}. According to \cite[p.~473, \S 21.43]{WW}
the function $\sigma(z)=\sigma(z|\omega_1,\omega_2)$ formed with
the two periods $2\omega_1$, $2\omega_2$, is expressible in the form
\begin{subequations}\label{eq:sigma0}
\begin{align}
  \sigma(z|\omega_1,\omega_2)
  &=\frac{2\omega_1}{\pi}\exp\!\Big(\frac{\eta z^2}{2\omega_1}\Big)
  \sin\!\Big(\frac{\pi z}{2\omega_1}\Big)
  \prod_{n=1}^\infty
  \frac{\big(1-2\tilde{p}^{2n}\cos\frac{\pi z}{\omega_1}+\tilde{p}^{4n}\big)}
  {(1-\tilde{p}^{2n})^2},\notag\\
  &=\omega_1 \,e^{\eta z^2/2\omega_1}
    \big(e^{\pi{\mathrm i}/2\omega_1}-e^{-\pi{\mathrm i}/2\omega_1}\big)
    \frac{(\tilde{p}^2e^{\pi{\mathrm i}z/\omega_1},
    \tilde{p}^2e^{-\pi{\mathrm i}z/\omega_1};\tilde{p}^2)_\infty}
    {\pi{\mathrm i}\,(\tilde{p}^2;\tilde{p}^2)_\infty^2}\notag\\
  &=\frac{{\mathrm i}\,\omega_1e^{(\eta z^2-\pi{\mathrm i})/2\omega_1}\,
    \theta(e^{\pi{\mathrm i}/2\omega_1};\tilde{p}^2)}
    {\pi\,(\tilde{p}^2;\tilde{p}^2)_\infty^2},
\end{align}
where
\begin{equation}
  \eta=\frac{\pi^2}{12\omega_1}
  \Big(1-24\sum_{n=1}^\infty\frac{\tilde{p}^{2n}}{(1-\tilde{p}^{2n})^2}\Big)
  \quad\;\text{and}\quad\;
  \tilde{p}=e^{\frac{\pi{\mathrm i}\omega_2}{\omega_1}}.
\end{equation}
\end{subequations}

From \eqref{eq:theta-addition} it is immediate that the sigma function
satisfies the addition formula (cf.\ \cite[p.~451, Ex.~5]{WW})
\begin{align}\label{eq:sigma-addition}
  \sigma(x+y)\sigma(x-y) \sigma(u+t)\sigma(u-t)-
  \sigma(x+t)\sigma(x-t) \sigma(u+y)\sigma(u-y)&\notag\\
  =\sigma(y+t)\sigma(y-t) \sigma(x+u)\sigma(x-u)&.
\end{align}
Now let $\zeta(z):=\frac{\mathrm d}{{\mathrm d}z}\log\sigma(z)$ and
$\wp(z):=-\frac{\mathrm d}{{\mathrm d}z}\zeta(z)$ (called the
\textit{Weierstra{\ss} $\zeta$-function} and \textit{Weierstra{\ss}
  $\wp$-function}, respectively; the latter is actually an
elliptic function). Differentiation of both sides of
\eqref{eq:sigma-addition} with respect to $x$,
followed by division by $\sigma(x+y)\sigma(x-y) \sigma(u+t)\sigma(u-t)$
and putting $u=x$, gives (cf.\ \cite[p.~461, Ex.~38]{WW})
\begin{equation}\label{sigma-derivative0}
  \zeta(x+y)+\zeta(x-y)-\zeta(x+t)-\zeta(x-t)=
  \frac{\sigma'(0)\sigma(2x)\sigma(y+t)\sigma(y-t)}
  {\sigma(x+y)\sigma(x-y)\sigma(x+t)\sigma(x-t)}.
\end{equation}
Now it is not difficult to show directly, using \eqref{eq:sigma0}, that
$\sigma'(0)=1$. After making the simultaneous substitutions
$(x,y,t)\mapsto((u+v)/2,(u-v)/2,-w-(u+v)/2)$ in
\eqref{sigma-derivative0} we obtain
\begin{equation}\label{sigma-derivative}
  \zeta(u)+\zeta(v)+\zeta(w)-\zeta(u+v+w)=
  \frac{\sigma(u+v)\sigma(u+w)\sigma(v+w)}
  {\sigma(u)\sigma(v)\sigma(w)\sigma(u+v+w)},
\end{equation}
which after the simultaneous substitutions
$(u,v,w)\mapsto(2u,2v,-u-v)$ reduces to 
\begin{equation}\label{sigma-derivative2}
  \zeta(2u)+\zeta(2v)-2\zeta(u+v)=
  \frac{\sigma(2u+2v)\sigma^2(u-v)}
  {\sigma(2u)\sigma(2v)\sigma^2(u+v)}.
\end{equation}
Finally, by differentiating \eqref{sigma-derivative}
with respect to $x$ and applying \eqref{sigma-derivative} to the result,
one obtains, after another substitution of variables,
(cf.\ \cite[p.~451, Ex.~1]{WW})
\begin{equation}\label{eq:wp}
\wp(v)-\wp(u)=\frac{\sigma(u-v)\sigma(u+v)}{\sigma^2(u)\sigma^2(v)}.
\end{equation}
This important formula can be used to prove the addition formula
in \eqref{eq:sigma-addition} without knowing  the definition of the
Weierstra{\ss} $\wp$-function. While we will actually not use this function
and the formula in \eqref{eq:wp} in this paper, we nevertheless include
this very classical material to make our survey in this section more complete.

In the following we will choose as half-periods $\omega_1=\frac 12$,
$\omega_2=\frac{\tau}2$ (with $\tau\in{\mathrm i}\R^+$), and put
$p=\tilde{p}^2$, so $p=e^{2\pi{\mathrm i}\tau}$. (Notice that the condition
on $\tau$ guarantees that $0<|p|<1$.)
With these specializations we have
\begin{subequations}
  \begin{equation}\label{eq:sigma}
    \sigma(z)=\sigma\Big(z\,\Big|\frac 12,\frac{\tau}2\Big)
    =\frac{{\mathrm i}\,e^{\eta z^2-\pi{\mathrm i z}}\,
  \theta(e^{2\pi{\mathrm i}z};p)}
  {2\pi\,(p;p)_\infty^2},
\end{equation}
where
\begin{equation}
  \eta=\frac{\pi^2}6\Big(1-24\sum_{n=1}^\infty\frac{p^n}{(1-p^n)^2}\Big).
\end{equation}
\end{subequations}

We are now ready to define our elliptic numbers and elliptic binomial
coefficients. For $a,b,q,p\in\C$ with $|p|<1$ we define the elliptic
(or $a,b;q,p$-)extension of a number $x$ (which does not need to be an integer)
as follows:
\begin{equation}
[x]_{a,b;q,p}=\frac{\theta(q^x,aq^x,bq,aq/b;p)}{\theta(q,aq,bq^x,aq^x/b;p)}.
\end{equation}
(This definition corresponds to that for the elliptic numbers considered in
\cite{SchlosserYoo1,SchlosserYoo2,SchlosserYoo3,SchlosserYoo4}
subject to the substitution $b\mapsto bq^{-1}$.)

From \eqref{eq:theta-addition} we have for all
$x$ and $y$ the following addition formula
for the elliptic numbers:
\begin{subequations}\label{abq_addition-ell}
\begin{equation}
[x]_{a,b;q,p} + W_{a,b;q,p}(x) [y-x]_{aq^{2x},bq^x;q,p}
= [y]_{a,b;q,p},
\end{equation}
where $W_{a,b;q,p}(x)$ is the \textit{elliptic-weight}, defined by
\begin{equation}\label{single_weight-ell}
W_{a,b;q,p}(x) = \frac{\theta(aq^{1+2x},b,bq,a/b,aq/b;p)}
{\theta(aq,bq^x,bq^{1+x},aq^x/b,aq^{1+x}/b;p)}q^x.
\end{equation}
\end{subequations}
It is clear that the properties \eqref{positivity_dom} and
\eqref{properties} readily extend to the elliptic level.

For parameter $a\in\C$, base $0<|q|<1$, nome $|p|<1$ and non-negative
integer $k$, the \textit{theta} (or \textit{$q,p$-})\textit{shifted factorial}
is defined as
\begin{equation*}
(a;q,p)_k:=\theta(a,aq,\dots,aq^{k-1};p),
\end{equation*}
cf.\ \cite[Ch.~11]{GasperRahman}.
For brevity, we use the following compact notation for products of theta
shifted factorials:
\begin{equation*}
(a_1,\dots,a_m;q,p)_k:=(a_1;q,p)_k\cdots(a_m;q,p)_k,\qquad\text{where}\quad
k\in\N_0.
\end{equation*}

Further, for $x,a,b,q,p\in\C$ with $0<|q|<1$, $|p|<1$, and a
non-negative integer $k$
we define the elliptic (or $a,b;q,p$-)binomial coefficient as
\begin{align}\label{a_b_q_p_binom}
\qbin{x}{k}_{a,b;q,p} :={}&
\frac{(q^{1+x-k},aq^{1+x-k},bq^{1+k},aq^{1-k}/b;q,p)_k}
{(q,aq,bq^{1+x},aq^{1+x-2k}/b;q,p)_k}.
\end{align}
If we assume $x$ and $k$ being nonnegative integers, and
replace $b$ by $bq$ on the right-hand side,
the definition in \eqref{a_b_q_p_binom}
matches that of the elliptic binomial coefficient as used
in \cite{Schlosser1, SchlosserYoo1, SchlosserYoo2}.
Clearly, $\qbin{x}{1}_{a,b/q;q,p}=[x]_{a,b;q,p}$.

\section{Log-concavity of elliptic numbers and elliptic binomial coefficients}
\label{sec:ell-results}

Just as the $a,b;q$-numbers satisfy an $a,b;q$-version of continuous 
strong log-concavity where certain shifts of the variables $a$ and $b$ occur
(see Theorem ~\ref{log_concavity_abq_addition}),
the elliptic numbers satisfy the following result:

\begin{theorem}\label{log_concavity_abq_addition_ell}
Let $q,p,x,y,r,a,b$ be real numbers satisfying $0<q<1$, $0<p<1$,
$x\geq y\geq r\geq 0$, and  $0<a<b<1$.
Then the elliptic numbers satisfy the inequality 
\begin{equation*}
[x]_{aq^{2r},bq^r;q,p}[y]_{a,b;q,p} \geq [x+r]_{a,b;q,p}[y-r]_{aq^{2r},bq^r;q,p}.
\end{equation*}
\end{theorem}
\begin{proof}
The proof is just like that of Theorem ~\ref{log_concavity_abq_addition},
but we use \eqref{abq_addition-ell} instead of  \eqref{abq_addition}.
\end{proof}

In the remaining section we will consider log-concavity results where the
variables $a$ and $b$ do \textit{not} shift.
Our results crucially depend on the following proposition.
\begin{proposition}\label{prop:theta}
  Let $q,p,x,r$ be real numbers satisfying $0<q<1$, $x\geq r>0$ and
  $0<p<q^{2r}$, and let the function $f$ be defined by  
\begin{equation}\label{eq:f_theta}
  f(u):=f_{q,p,x,r}(u)=\frac{\theta(uq^{x+r},uq^{x-r};p)}{\theta^2(uq^{x};p)}
  \qquad\text{for $u\in[\delta,\lambda]$},
\end{equation}
where $\delta=pq^{-x-r}$ and $\lambda=q^{r-x}$.
Then $f$ is continuous and positive on the interval $[\delta,\lambda]$,
with first and second derivatives being negative on the interval
$(\delta,\lambda)$.
\end{proposition}
\begin{proof}
  The continuity and positivity is clear. To compute the derivatives
  on $(\delta,\lambda)$,  we rewrite $f(u)$ in terms of the sigma function.
  Let $p=e^{2\pi{\mathrm i}\tau}$, $q=e^{2\pi{\mathrm i}\nu}$, and 
  $u=e^{2\pi{\mathrm i}z}$, where $\tau,\nu,z\in{\mathrm i}\R$.
  We have ${\mathrm d}u=2\pi{\mathrm i}\,e^{2\pi{\mathrm i}z}{\mathrm d}z$,
  or equivalently, $\frac{\mathrm d}{{\mathrm d}z}=2\pi{\mathrm i}u\,
  \frac{\mathrm d}{{\mathrm d}u}$, and, according to \eqref{eq:sigma},
  \begin{align*}
    f(e^{2\pi{\mathrm i}z})
    =
    \frac{\theta(e^{2\pi{\mathrm i}((x+r)\nu+z)},
    e^{2\pi{\mathrm i}((x-r)\nu+z)};p)}
    {\theta^2(e^{2\pi{\mathrm i}(x\nu+z)};p)}
    =e^{-2\eta r^2\nu^2}\,\frac{\sigma(z+(x+r)\nu)\sigma(z+(x-r)\nu)}
      {\sigma^2(z+x\nu)}.
\end{align*}
Thus, the first derivative of $f$ is
\begin{align*}
  \frac{\mathrm d}{{\mathrm d}u}\,f(u)
  &=
\frac{e^{-2\pi{\mathrm i}z}}{2\pi{\mathrm i}}\,
  \frac{\mathrm d}{{\mathrm d}z}\,f(e^{2\pi{\mathrm i}z})=
    \frac{e^{-2\pi{\mathrm i}z-2\eta r^2\nu^2}}{2\pi{\mathrm i}}\,
    \frac{\mathrm d}{{\mathrm d}z}\,
    \frac{\sigma(z+(x+r)\nu)\sigma(z+(x-r)\nu)}
    {\sigma^2(z+x\nu)}\\
  &=\frac{e^{-2\pi{\mathrm i}z-2\eta r^2\nu^2}}{2\pi{\mathrm i}\,\sigma^3(z+x\nu)}
    \big[\sigma(z+x\nu)\sigma(z+(x+r)\nu)\sigma'(z+(x-r)\nu)\\*
  &\qquad\qquad\qquad\qquad
    +\sigma(z+x\nu)\sigma'(z+(x+r)\nu)\sigma(z+(x-r)\nu)\\*
  &\qquad\qquad\qquad\qquad
    -2\sigma'(z+x\nu)\sigma(z+(x+r)\nu)\sigma(z+(x-r)\nu)\big]\\
  &=\frac{e^{-2\pi{\mathrm i}z}}{2\pi{\mathrm i}}f(e^{2\pi{\mathrm i}z})
    \big[\zeta(z+(x+r)\nu)+\zeta(z+(x-r)\nu)-2\zeta(z+x\nu)\big]\\
  &=\frac{e^{-2\pi{\mathrm i}z-2\eta r^2\nu^2}}{2\pi{\mathrm i}}
    \frac{\sigma(2(z+x\nu))\sigma^2(r\nu)}{\sigma^4(z+x\nu)}\\
  & =-q^{x-r}\frac{(p;p)^2_\infty\theta^2(q^r;p)\theta(u^2q^{2x};p)}
    {\theta^4(uq^x;p)},
  \end{align*}
  where we applied \eqref{sigma-derivative2} in the penultimate
  equality, and used \eqref{eq:sigma} to rewrite the expression
  in terms of the modified theta function in the last equality.
  The final expression obtained for the first derivative of $f(u)$
  is clearly negative. (The minus sign in front
  of the product stems from collecting all the factors $\mathrm i$
  from the applications of \eqref{eq:sigma}; indeed,
  ${\mathrm i}^3/{\mathrm i}^5=-1$.)
  
  We still need to compute the second derivative of $f$ and show that it
  is also negative.
 We have 
 \begin{align*}
   \frac{{\mathrm d}^2}{{\mathrm d}u^2}f(u)
   &=-q^{x-r}(p;p)^2_\infty\theta^2(q^r;p)\,
   \frac{\mathrm d}{{\mathrm d}u}\frac{\theta(u^2q^{2x};p)}
     {\theta^4(uq^x;p)}\\
      &=-q^{x-r}\frac{(p;p)^2_\infty\theta^2(q^r;p)\theta(u^2q^{2x};p)}
     {\theta^4(uq^x;p)}
        \Big(\frac{\mathrm d}{{\mathrm d}u}\log\theta(u^2q^{2x};p)-
        4\frac{\mathrm d}{{\mathrm d}u}\log\theta(uq^x;p)\Big).
 \end{align*}
 All we need to do is to show that
 \begin{equation*}
   \frac{\mathrm d}{{\mathrm d}u}\log\theta(u^2q^{2x};p)>
        4\frac{\mathrm d}{{\mathrm d}u}\log\theta(uq^x;p).
   \end{equation*}
   Since the logarithmic derivative of a product is the sum
   of the logarithmic derivatives, we thus need to show
\begin{align*}
&\sum_{j=0}^\infty\frac{\mathrm d}{{\mathrm d}u}\log
  \big[(1-p^ju^2q^{2x})+(1-p^{j+1}u^{-2}q^{-2x})\big]\\
  &>
4\sum_{j=0}^\infty\frac{\mathrm d}{{\mathrm d}u}\log
\big[(1-p^juq^{x})+(1-p^{j+1}u^{-1}q^{-x})\big].
\end{align*}
The inequality actually holds term-wise, for each $j$.
Indeed, comparison of the logarithmic derivatives for fixed $j$ amounts to
\begin{equation*}
  \frac{-2p^juq^{2x}}{1-p^ju^2q^{2x}}+
  \frac{4p^{j+1}u^{-3}q^{-2x}}{1-p^{j+1}u^{-2}q^{-2x}}>
   \frac{-4p^jq^{x}}{1-p^juq^{x}}+
  \frac{4p^{j+1}u^{-2}q^{-x}}{1-p^{j+1}u^{-1}q^{-x}},
\end{equation*}
which is easy to verify, as already the first summand on the left-hand
side is greater than the first summand on the right-hand side,
which analogously holds for the second summands on both sides.
\end{proof}

Having set up all the ingredients, the results now
immediately extend from the $a,b;q,p$-case to the elliptic case.

\begin{theorem}\label{direct_log_concavity_ell}
The elliptic numbers $[x]_{a,b;q,p}$ are continuously
strongly log-concave. In particular, for all real numbers $q,p,x,y,r,a,b$
satisfying $0<q<1$, $0< p<q^{2r}$, $x\geq y\geq r\geq 0$, 
and $pq^{-x-r}<a<b<1$ we have
\begin{equation*}
[x]_{a,b;q,p}[y]_{a,b;q,p}\geq [x+r]_{a,b;q,p}[y-r]_{a,b;q,p}.
\end{equation*}
\end{theorem}

\begin{proof}
  The proof follows the lines of that of the proof of
  Theorem~\ref{direct_log_concavity}
but we use, instead of \eqref{f_func},
$f(u)$ as defined in \eqref{eq:f_theta}, where $u\in[\delta,1]$
with $\delta=pq^{-x-r}$.
As a consequence of Proposition~\ref{prop:theta}, $f$ satisfies on
$[\delta,1]$ (which is contained in $[\delta,\lambda]$ with $\lambda=q^{r-x}$)
the necessary requirements for application of Lemma~\ref{lem:mult_Turan},
which we invoke with $\lambda=1$.
\end{proof}

\begin{theorem}\label{binom_top_log_concavity_ell}
For any real $q,p,a,b,x,y$ and non-negative integer $k$ satisfying $0<q<1$,
$0<p<q^2$, $pq^{-x-1}<a\leq bq^k<1$, and $x\geq y\ge 1$,
with the difference $x-y$ being
a non-negative integer, we have the discrete strong log-concavity
of elliptic binomial coefficients
\begin{equation*}\
\qbin{x}{k}_{a,b;q,p}\qbin{y}{k}_{a,b;q,p} \geq
\qbin{x+1}{k}_{a,b;q,p}\qbin{y-1}{k}_{a,b;q,p}.
\end{equation*}
\end{theorem}

\begin{proof}
The proof follows the lines of that of the proof of
Theorem~\ref{binom_top_log_concavity}
but we use, instead of \eqref{f_func},
$f(u)$ as defined in \eqref{eq:f_theta}, where $u\in[\delta,1]$
with $\delta=pq^{-x-1}$, for application of Lemma~\ref{lem:mult_Turan}.
\end{proof}

\section{Conclusion}
\label{sec:con}

We have shown that fundamental log-concavity results which are known
to hold for (the $q$-)numbers and ($q$-)binomial coefficients can be extended
to the $a,b;q$-case and even to the elliptic case.
Further closely related questions remain open,
for instance about the log-concavity of $a,b;q$- or elliptic extensions
of those other sequences considered by Sagan 
in \cite{Sagan2} or of the $a,b;q$- or elliptic rook numbers
(considered by the first author and Yoo~\cite{SchlosserYoo2,SchlosserYoo3}).

\section*{Acknowledgements}
All three authors thank the FWF Austrian Science Fund for generous support.
In particular, the first author was partially supported by
grants F50-08 and P32305,
the second author was fully supported by grants  F50-08 and F50-10,
and the third author was fully supported by grants F50-07, F50-09 and F50-11.

\end{document}